\documentclass{amsart}
\usepackage{mathptmx}
\usepackage{pxfonts}
\usepackage{amssymb}
\usepackage{amsfonts}
\usepackage{amsmath}
\usepackage{graphicx}
\usepackage{shadow}
\usepackage{color}
\usepackage[all]{xy}
\usepackage[pagebackref]{hyperref}
\newtheorem{thm}{Theorem}[section]
\newtheorem{corollary}[thm]{Corollary}
\newtheorem{lemma}[thm]{Lemma}

\newtheorem{claim}{Claim}

\theoremstyle{definition}

\newtheorem{example}[thm]{Example}

\theoremstyle{remark}
\newtheorem{remark}[thm]{Remark}

\DeclareMathOperator{\nil}{Nil}

\DeclareMathOperator{\Max}{Max}

\DeclareMathOperator{\qf}{qf}

\DeclareMathOperator{\Ker}{Ker}

\DeclareMathOperator{\Spec}{Spec}

\DeclareMathOperator{\wdim}{w.dim}

\DeclareMathOperator{\m}{\mathfrak{m}}
\DeclareMathOperator{\M}{\mathfrak{M}}
\DeclareMathOperator{\id}{id}
\newcommand{\field}[1]{\mathbb{#1}}

\newcommand{\Z }{\field{Z}}

\begin{document}
\baselineskip=15pt

\title[Bi-amalgamations subject to the arithmetical property]{Bi-amalgamations subject to the arithmetical property $^{(\star)}$}
\thanks{$^{(\star)}$ Supported by King Fahd University of Petroleum \& Minerals  under Research Grant \# RG1310.}

\author[S. Kabbaj]{S. Kabbaj $^{(1)}$}\thanks{$^{(1)}$ Corresponding author.}
\address{Department of Mathematics and Statistics, King Fahd University of Petroleum \& Minerals, Dhahran 31261, KSA}
\email{kabbaj@kfupm.edu.sa}

\author[N. Mahdou]{N. Mahdou}
\address{Department of Mathematics, Faculty of Science and Technology, University S. M.
Ben Abdellah, Fez 30000, Morocco}
\email{mahdou@hotmail.com}

\author[M. A. S. Moutui]{M. A. S. Moutui}
\address{Department of Mathematics, Faculty of Science and Technology, University S. M.
Ben Abdellah, Fez 30000, Morocco}
\email{moutu\_2004@yahoo.fr}

\date{\today}


\begin{abstract}
This paper establishes necessary and sufficient conditions for a bi-amalgamation to inherit the arithmetical property, with applications on the weak global dimension and transfer of the semihereditary property. The new results compare to previous works carried on various settings of duplications and amalgamations, and capitalize on recent results on bi-amalgamations. All results are backed with new and illustrative examples arising as bi-amalgamations.
\end{abstract}
\maketitle

\section{Introduction}

 Throughout, all rings considered are commutative with unity and all modules are unital. Let $f: A\rightarrow B$ and $g: A\rightarrow C$ be two ring homomorphisms and let $J$ and $J'$ be two ideals of $B$ and $C$, respectively, such that $I_{o}:=f^{-1}(J)=g^{-1}(J')$. The \emph{bi-amalgamation} (or \emph{bi-amalgamated algebra}) of $A$ with $(B, C)$ along $(J, J')$ with respect to $(f,g)$ is the subring of $B\times C$ given by
$$A\bowtie^{f,g}(J,J'):=\big\{(f(a)+j,g(a)+j') \mid a\in A, (j,j')\in J\times J'\big\}.$$
This construction was introduced in \cite{KLT} as a natural generalization of duplications \cite{CJKM,D,DF1,DF2,MY,Sh} and amalgamations \cite{DFF1,DFF2,Fin}. In \cite{KLT}, the authors provide original examples of bi-amalgamations and, in particular, show that Boisen-Sheldon's CPI-extensions \cite{BSh2}  can be viewed as bi-amalgamations (Notice that \cite[Example 2.7]{DFF1} shows that CPI-extensions can be viewed as quotient rings of amalgamated algebras). They also show how every bi-amalgamation can arise as a natural pullback (or even as a conductor square) and then characterize pullbacks that can arise as bi-amalgamations. This allowed them to characterize Traverso's glueings of prime ideals \cite{Ped,Tam,Tra,Yan} which can be viewed as special bi-amalgamations. Then, the last two sections deal, respectively, with the transfer of some basic ring theoretic properties to bi-amalgamations and the study of their prime ideal structures.  All their results recover known results on duplications and amalgamations. Finally, it is worthwile recalling that, very recently, Finocchiaro investigated necessary and sufficient conditions for an amalgamated algebra to inherit various Pr\"ufer conditions (including the arithmetical property) \cite{Fin}.

This paper studies the transfer of the arithmetical property and related notions to bi-amalgamations. A ring $R$ is \emph{arithmetical} if every finitely generated ideal is locally principal \cite{Fu,J};  and $R$ is \emph{semihereditary} if every finitely generated ideal is projective \cite{CE}. The following diagram of implications summarizes the relations between the main three notions involved in this paper:

 {
 \begin{center}
 \begin{tabular}{ccccc}
                        &           &$R$ is semihereditary      &           &\\
                        &           &$\Downarrow$               &           &\\
                        &           &$\wdim(R)\leq 1$           &           &\\
                        &           &$\Downarrow$               &           &\\
                        &           &$R$ is arithmetical        &           &

 \end{tabular}
 \end{center}}

\noindent where $\wdim(R)$ denotes the weak global dimension of $R$. Recall that all these properties are identical to the notion of Pr\"ufer domain if $R$ has no zero-divisors,
and that the above implications are irreversible, in general, as shown by examples provided in \cite{AJK,BKM,BG,BG2,Bo2007,Bo2008,Bo2011,CJKM,Cou,G2,G3,GS}. Very recently,
these conditions (among other Pr\"ufer conditions) were thoroughly investigated in various contexts of duplications \cite{CJKM}.

This paper establishes necessary and sufficient conditions for a bi-amalgamation to inherit the arithmetical property, with applications on the weak global dimension and transfer of the semihereditary property. Section 2 is devoted to the transfer results, the main of which (Theorem~\ref{arith}) states that "\emph{ $A\bowtie^{f,g}(J,J')$ is arithmetical if and only if both $f(A)+J$ and $g(A)+J'$ are arithmetical and, for every $\m\in\Max(A,I_{o})$, $J_{S_{\m}}=0$ or $J'_{S'_{\m}}=0$}". Follow several applications featuring the transfer of other related properties in bi-amalgamations, amalgamations, and duplications. All obtained results recover and compare to previous works carried on various settings of duplications and amalgamations, and capitalize on recent results on bi-amalgamations. (As to previous works on pullbacks, see Remark~\ref{arith-r1}.) In Section 3, all results are backed with new and illustrative examples arising as bi-amalgamations.

Notice, at this point, that in the presence of the equality $f^{-1}(J)=g^{-1}(J')$, $J=B$ if and only if $J'=C$; and in this case $A\bowtie^{f,g}(J,J')= B\times C$. Therefore, in this paper, we will omit this case (i.e., $J$ and $J'$ will always be proper) since it is known that the above Pr\"ufer notions are stable under finite products (cf. \cite[p. 251]{CJKM}).

Throughout, for a ring $R$, $\Spec(R)$ (resp., $\Max(R)$)  will denote the set of all prime (resp., maximal) ideals of $R$, and, for any ideal $I$ of $R$, $\Spec(R,I)$ (resp., $\Max(R,I)$) will denote the set of all prime (resp., maximal) ideals of $R$ containing $I$.

\section{Results}\label{sec:2}

\noindent Let $f: A\rightarrow B$ and $g: A\rightarrow C$ be two ring homomorphisms and let $J$ and $J'$ be two \emph{proper} ideals of $B$ and $C$, respectively, such that $I_{o}:=f^{-1}(J)=g^{-1}(J')$. All along this section, $A\bowtie^{f,g}(J,J')$ will denote the bi-amalgamation of $A$ with  $(B, C)$ along $(J, J')$ with respect to $(f,g)$. This section investigates the arithmetical and semihereditary properties  as well as the weak global dimension in bi-amalgamations.

The first main result establishes necessary and sufficient conditions for a bi-amalgamation to inherit the arithmetical property. To this purpose, let us adopt the following notation:\\
For any $p\in\Spec(A,I_{o})$ (resp., $\in\Max(A,I_{o})$), consider the multiplicative subsets
$$S_{p}:=f(A-p)+J\ \text{ and }\ S'_{p}:=g(A-p)+J'$$
of $B$ and $C$, respectively, and let
$$f_p: A_p\rightarrow B_{S_{p}}\ \text{ and }\ g_p: A_p\rightarrow C_{S'_{p}}$$
be the canonical ring homomorphisms induced by $f$ and $g$. One can easily check that
$$f_p^{-1}(J_{S_{p}})=g_p^{-1}(J'_{S'_{p}})=(I_{o})_p.$$ Moreover, by \cite[Lemma 5.1]{KLT}, $P:=p\bowtie^{f,g}(J,J')$ is a prime (resp., maximal) ideal of $A\bowtie^{f,g}(J,J')$ and, by \cite[Proposition 5.7]{KLT}, we have

$$\big(A\bowtie^{f,g}(J,J')\big)_{P}\cong A_p\bowtie^{f_p,g_p}(J_{S_{p}},J'_{S'_{p}}).$$
These facts will be used in the sequel without explicit reference. Recall that a local arithmetical ring is also called a chained ring (i.e., its lattice of ideals is totally ordered by inclusion).

\begin{thm}\label{arith}
Under the above notation, we have:
\begin{enumerate}
\item $A\bowtie^{f,g}(J,J')$ is a chained ring if and only if both $f(A)+J$ and $g(A)+J'$ are chained rings and $J=0$ or $J'=0$.
\item   $A\bowtie^{f,g}(J,J')$ is arithmetical if and only if both $f(A)+J$ and $g(A)+J'$ are arithmetical and, for every $\m\in\Max(A,I_{o})$, $J_{S_{\m}}=0$ or $J'_{S'_{\m}}=0$.
\end{enumerate}
\end{thm}

\begin{proof} (1) By \cite[Proposition 3.1]{KLT}, $A\bowtie^{f,g}(J,J')$ arises as a pullback $D:=\alpha\times_{\frac{A}{I_{o}}}\beta$ with $\Ker(\alpha)=J$, $\Ker(\beta)=J'$, and $p_{1}(D)=f(A)+J$ (resp., $p_{2}(D)= g(A)+J'$) where $p_{1}$ (resp., $p_{2}$) denotes the restriction to $D$ of the projection of $(f(A)+J) \times (g(A)+J')$ into $f(A)+J$ (resp., $g(A)+J'$). Moreover, recall that the chained ring notion is stable under factor ring and $\frac{f(A)+J}{J}\cong \frac{g(A)+J'}{J'}$ \cite[Proposition 4.1(3)]{KLT}). Therefore, the result follows readily from \cite[Proposition 4.9]{Fin}.

(2) First note that (1) is the local version of (2). In order to see this, recall that $A\bowtie^{f,g}(J,J')$ is local if and only if both $f(A)+J$ and $g(A)+J'$ are local; and $\m\bowtie^{f,g}(J,J')$ is the maximal ideal of $A\bowtie^{f,g}(J,J')$, where $\m$ is the unique maximal ideal of $A$ containing $I_{o}$ \cite[Proposition 5.4]{KLT}. Moreover, in view of the isomorphism $$\dfrac{A\bowtie^{f,g}(J,J')}{0\times J'}\cong f(A)+J$$
given by \cite[Proposition 4.1(2)]{KLT}, we deduce that
$$\dfrac{\m\bowtie^{f,g}(J,J')}{0\times J'}\simeq f(\m)+J$$
is the maximal ideal of $f(A)+J$ and, similarly,  $g(\m)+J'$ is the maximal ideal of $g(A)+J'$. Hence $S_{\m}$ and $S'_{\m}$ shall consist of units only so that  $J_{S_{\m}}=J$ and $J'_{S'_{\m}}=J'$, as desired.

Now, assume that $A\bowtie^{f,g}(J,J')$ is arithmetical. Let $\m\in\Max(A,I_{o})$ and $\M:=\m\bowtie^{f,g}(J,J')$. Therefore $$A_{\m}\bowtie^{f_{\m},g_{\m}}(J_{S_{\m}},J'_{S'_{\m}})\cong \big(A\bowtie^{f,g}(J,J')\big)_{\M}$$
 is a chained ring  (since the arithmetical property is stable under localization). By (1),  $J_{S_{\m}}$ or $J'_{S'_{\m}}$ is null, as desired. Next, let $L\in\Spec(f(A)+J)$ and consider the prime ideal of $A\bowtie^{f,g}(J,J')$ given by
 $$\overline{L}:=\big(L\times (g(A)+J')\big)\cap\big(A\bowtie^{f,g}(J,J')\big).$$
 If $J\nsubseteq L$, then by \cite[Proposition 5.3(2)]{KLT}, $$(f(A)+J)_L\cong \left(A\bowtie^{f,g}(J,J')\right)_{\overline{L}}$$ is arithmetical. Next, assume that $J\subseteq L$. By \cite[Lemma 5.2]{KLT}, we have
 $$\overline{L}:=p\bowtie^{f,g}(J,J')$$
 where $p:=f^{-1}(L)\in\Spec(A,I_{o})$ and one can easily verify that $L=f(p)+J$. So,
 $$A_p\bowtie^{f_p,g_p}(J_{S_{p}},J'_{S'_{p}})\cong\big(A\bowtie^{f,g}(J,J')\big)_{\overline{L}}$$
  is a chained ring. By (1), $f_{p}(A_{p})+J_{S_{p}}$ is a chained ring and (as seen above) with maximal ideal $f_{p}(pA_{p})+J_{S_{p}}$. The fact that $f_{p}(A_{p})+J_{S_{p}}$ is local yields:
\begin{claim}\label{claim1}
$f_{p}(A_{p})+J_{S_{p}}=(f(A)+J)_L$.
\end{claim}
Indeed, first observe that
$$S_{p}=(f(A)+J)\setminus (f(p)+J)=(f(A)+J)\setminus L$$
and, hence, both  $f_{p}(A_{p})+J_{S_{p}}$ and $(f(A)+J)_L$ are subrings of $B_{S_{p}}$. The forward inclusion is obvious. To prove the other, let $x\in(f(A)+J)_L$; that is,
$$x=\dfrac{f(a)+i}{f(s)+j}=\left(\dfrac{1}{f(s)+j}\right)\left(\dfrac{f(a)}{1}\right)+\dfrac{i}{f(s)+j}$$
for some $a\in A$, $s\in A\setminus p$ and $i,j\in J$. Clearly, it suffices to show that $$\dfrac{1}{f(s)+j}\in f_{p}(A_{p})+J_{S_{p}}.$$ This is true since one can check that $$\dfrac{f(s)+j}{1}=\dfrac{f(s)}{1}+\dfrac{j}{1}\notin f_{p}(pA_{p})+J_{S_{p}}$$
as $f_{p}(pA_{p})+J_{S_{p}}$ is the maximal ideal of $f_{p}(A_{p})+J_{S_{p}}$, proving the claim.\\
By (1),  $(f(A)+J)_L$ is arithmetical. Consequently,  $f(A)+J$ is (locally) arithmetical and so is $g(A)+J'$ via similar arguments.

Conversely, assume  $f(A)+J$ and $g(A)+J'$ are arithmetical and, $\forall\m\in\Max(A,I_{o})$, $J_{S_{\m}}$ or $J'_{S'_{\m}}$ is null. Let $\M\in\Max(A\bowtie^{f,g}(J,J'))$. Suppose that $J\times J'\nsubseteq \M$. By \cite[Proposition 5.3(2)]{KLT}, there is $L$, say, in $\Spec(f(A)+J)$ such that
$$(A\bowtie^{f,g}(J,J'))_{\M}\cong (f(A)+J)_L.$$
So, in this case, $(A\bowtie^{f,g}(J,J'))_{\M}$ is obviously arithmetical. Next, suppose $J\times J'\subseteq \M$. By \cite[Proposition 5.3(1) \& Lemma 5.1]{KLT}, there is a unique $\m\in\Max(A,I_{o})$ such that
$$\M=\m\bowtie^{f,g}(J,J').$$
By hypothesis we have, say, $J'_{S'_{\m}}=0$. Now, let  $L:= f(\m)+J$, a prime ideal of $f(A)+J$. It follows, via \cite[Proposition 4.1(2)]{KLT} and Claim~\ref{claim1}, that:
$$
\begin{array}{lll}
\big(A\bowtie^{f,g}(J,J')\big)_{\M}         &\cong  &A_{\m}\bowtie^{f_{\m},g_{\m}}(J_{S_{\m}},0)\\
                                            &\cong  &f_{\m}(A_{\m})+J_{S_{\m}}\\
                                            &=      &(f(A)+J)_L.
\end{array}
$$
So, in this case too, $(A\bowtie^{f,g}(J,J'))_{\M}$ is arithmetical. Consequently, $A\bowtie^{f,g}(J,J')$ is arithmetical,  completing the proof of the theorem.
\end{proof}

\begin{remark}\label{arith-r1}
In Proposition 3.2 of \cite{KLT}, it is proved that every bi-amalgamation $A\bowtie^{f,g}(J,J')$ can be viewed as a conductor square with conductor $J\times J'$. Boynton examined the transfer of the arithmetical property to conductor squares in the special case where the conductor ideal is regular \cite[Theorem 3.3]{Bo2007} (and also \cite[Theorem 4.1]{Bo2008}). We cannot appeal to this result in the context of Theorem~\ref{arith} since, under the assumption ``$J\times J'$ is regular,"  the bi-amalgamation $A\bowtie^{f,g}(J,J')$ can never satisfy the arithmetical property (because of the necessary condition: $\forall\m\in\Max(A,I_{o})$, $J_{S_{\m}}=0$ or $J'_{S'_{\m}}=0$). This remark is also valid for the forthcoming Corollary~\ref{weak} (on the weak global dimension) and Corollary~\ref{semi} (on the semi-hereditary property).
\end{remark}

\begin{remark}\label{arith-r2}
Observe that Theorem~\ref{arith}(1) can also read as follows: \emph{$A\bowtie^{f,g}(J,J')$ is a chained ring if and only if ``$J=0$ and $g(A)+J'$ is a chained ring" or ``$J'=0$ and $f(A)+J$ is a chained ring}" which is obvious from the facts that the chained ring notion is stable under factor ring and $\frac{f(A)+J}{J}\cong \frac{g(A)+J'}{J'}$ \cite[Proposition 4.1(3)]{KLT}).

\end{remark}

As an illustrative example for Theorem~\ref{arith}, Example~\ref{arith-ex1} provides an original arithmetical ring which arises as a bi-amalgamation.

Recall that the amalgamation of $A$ with $B$ along $J$ with respect to $f$ is given by
$$A\bowtie^{f} J:=\big\{(a,f(a)+j)\mid a\in A, j\in J\big\}.$$
Clearly, every amalgamation can be viewed as a special bi-amalgamation, since $A\bowtie^fJ= A\bowtie^{{\id_{A}}, f}(f^{-1}(J),J)$.
Accordingly, Theorem~\ref{arith} covers the special case of amalgamations, as recorded below.

\begin{corollary}\label{arith-c1}
Under the above notation, we have:
\begin{enumerate}
\item $A\bowtie^{f}J$ is a chained ring if and only if both $A$ and $f(A)+J$ are chained rings and $J=0$ or $f^{-1}(J)=0$.
\item   $A\bowtie^{f}J$ is arithmetical if and only if both  $A$ and $f(A)+J$ are arithmetical and, for every $\m\in\Max(A,f^{-1}(J))$, $J_{S_{\m}}=0$ or $f_{\m}^{-1}\left(J_{S_{\m}}\right)=0$.
\end{enumerate}
\end{corollary}

\begin{remark}\label{arith-r2-1}
Recently, Finocchiaro proved the following result for the transfer of the arithmetical property to amalgamations: "\emph{Assume that, for each $\m\in\Max(A, f^{-1}(J))$, either $f_{\m}$ is surjective or $f_{\m}^{-1}\left(J_{S_{\m}}\right)\not=0$. Then, $A\bowtie^{f}J$ is arithmetical if and only if $A$ is arithmetical, $J_{S_{\m}}=0$ for each  $\m\in\Max(A,f^{-1}(J))$, and for any $\m'\in\Max(B)$ not containing $J$, the ideals of $B_{\m'}$ are totally ordered by inclusion}" \cite[Proposition 4.10]{Fin}. Corollary~\ref{arith-c1} covers this result due to the basic fact that if $f_{\m}$ is surjective and $J_{S_{\m}}\not=0$ then $f_{\m}^{-1}\left(J_{S_{\m}}\right)\not=0$; combined with the two-type maximal ideal structure of $A\bowtie^{f}J$ (cf.  \cite[Proposition 2.6]{DFF2} and \cite[Proposition 2.5]{Fin}); precisely,  $A_{\m}\cong (A\bowtie^{f}J)_{\m\bowtie^{f}J}$ when $J_{S_{\m}}=0$ and $B_{\m'}\cong (A\bowtie^{f}J)_{\overline{\m'}}$ where $\overline{\m'}=\big\{(a, f(a)+j)\in A\bowtie^{f}J \mid f(a)+j\in \m'\big\}$.
\end{remark}

\begin{remark}\label{arith-r3}
Assume $J\not=0$. Then Remark~\ref{arith-r2} combined with \cite[Corollary 4.6]{KLT} yield: \emph{$A\bowtie^{f}J$ is a chained ring (resp., valuation domain) if and only if $f^{-1}(J)=0$ and $f(A)+J$ is a chained ring (resp., valuation domain)}.
\end{remark}

For an original example of arithmetical ring arising as an amalgamation, see Example~\ref{arith-c1-ex1}. Next, let $I$ be a \emph{proper} ideal of $A$.  The (amalgamated) duplication of $A$ along $I$ is a special amalgamation given by $$A\bowtie I:=A\bowtie^{\id_{A}} I=\big\{(a,a+i)\mid a\in A, i\in I\big\}.$$
The above corollary recovers known results on the transfer of the arithmetical property to duplications, as shown below.

\begin{corollary}[{\cite[Theorem 3.2(1) \& Corollary 3.8(1)]{CJKM}}]\label{arith-c2}\label{arith-c2}
We have:
\begin{enumerate}
  \item $A\bowtie I$  is a chained ring if and only if $A$ is a chained ring and $I=0$.
  \item $A\bowtie I$  is arithmetical if and only if $A$ is arithmetical and $I_{\m}=0$, $\forall\m\in\Max(A,I)$.
\end{enumerate}
\end{corollary}

As another application of Theorem~\ref{arith}, we get necessary and sufficient conditions for a bi-amalgamation to have weak global dimension at most 1. For this purpose, let $\nil(R)$ denote the nilradical of a ring $R$.

\begin{corollary}\label{weak}
Assume $\wdim(f(A)+J)\leq 1$, $\wdim(g(A)+J')\leq 1$, $J\cap \nil(B)=0$, $J'\cap \nil(C)=0$ and, $\forall\m\in\Max(A,I_{o})$, $J_{S_{\m}}=0$ or $J'_{S'_{\m}}=0$. Then $\wdim\left(A\bowtie^{f,g}(J,J')\right)\leq 1$. The converse holds if $I_{o}$ is radical.
\end{corollary}

\begin{proof}
Recall that a ring $R$ has weak global dimension at most 1 if and only if $R$ is arithmetical and reduced \cite[Theorem 3.5]{BG}. A combination of this fact with Theorem~\ref{arith} and \cite[Proposition 4.7]{KLT} (on the transfer of the reduced property) leads to the conclusion.
\end{proof}

The converse of Corollary~\ref{weak} is not true in general. A counter-example in the special case of amalgamations is given in Example~\ref{weak-ex1}. Also, as an illustrative example for this result, Example~\ref{weak-ex2} features an original example of a ring with weak global dimension $\leq 1$ which arises as a bi-amalgamation.

For the special case of amalgamations, we get a more well-rounded result:

\begin{corollary}\label{arith-c3}
Under the above notation, we have:\\
 $\wdim(A\bowtie^{f}J)\leq 1$ if and only if $\wdim(A)\leq 1$, $f(A)+J$ is arithmetical (resp., Gaussian), $J\cap \nil(B)=0$ and, for every $\m\in\Max(A,f^{-1}(J))$, $J_{S_{\m}}=0$ or $f_{\m}^{-1}\left(J_{S_{\m}}\right)=0$.
\end{corollary}

\begin{proof}
Combine Corollary~\ref{arith-c1}(2) and \cite[Corollary 4.9]{KLT} with the well-known facts that $\wdim(R)\leq 1$ if and only if $R$ is reduced and arithmetical (resp., reduced and Gaussian) \cite[Theorems 3.5 \& 4.8]{BG}.
\end{proof}

Remark~\ref{arith-r2-1} is also valid for Corollary~\ref{arith-c3} and \cite[Proposition 4.11]{Fin} on the weak global dimension. See Example~\ref{weak-ex1} where  $\wdim(A\bowtie^{f}J)\leq 1$ and $\wdim(f(A)+J)> 1$. Corollary~\ref{arith-c3} recovers a known result for duplications:

\begin{corollary}[{\cite[Theorem 4.1(1)]{CJKM}}]\label{arith-c3-1}
We have:\\
$\wdim(A\bowtie I)\leq 1$ if and only if $\wdim(A)\leq 1$ and $I_{\m}=0,\ \forall\ \m\in\Max(A,I)$.
\end{corollary}

 A ring $R$ is semihereditary if and only if $R$ is coherent and $\wdim(R)\leq 1$ \cite[Theorem 3.3]{BG}. A combination of this fact with Corollary~\ref{weak} and \cite[Proposition 4.2]{KLT} establishes the transfer of the semihereditary property to bi-amalgamations in the special case of Noetherian settings.

\begin{corollary}\label{semi}
Assume that $f(A)+J$ and $g(A)+J'$ are Noetherian semihereditary rings, $J\cap \nil(B)=0$, $J'\cap \nil(C)=0$ and, $\forall\m\in\Max(A,I_{o})$, $J_{S_{\m}}=0$ or $J'_{S'_{\m}}=0$. Then, $A\bowtie^{f,g}(J,J')$ is a Noetherian semihereditary ring. The converse holds if $I_{o}$ is radical.
\end{corollary}

For the special case of amalgamations, we have a more well-rounded general result. For this purpose, we first recall the following result which examines the transfer of coherence to amalgamations.

\begin{lemma}[{\cite[Theorem 2.2]{AM}}]\label{coh}
Suppose that $f^{-1}(J)$ and $J$ are finitely generated in $A$ and $f(A)+J$, respectively. Then:  $A\bowtie^{f}J$ is coherent if and only if $A$ and $f(A)+J$ are coherent.
\end{lemma}

Contrast this result with \cite[Proposition 4.14]{Fin} on the transfer of coherence to amalgamations. The next result is a combination of this lemma and Corollary~\ref{arith-c3} with the well-known fact that $R$ is semihereditary if and only if $R$ is coherent and $\wdim(R)\leq 1$ \cite[Theorem 3.3]{BG}.

\begin{corollary}\label{semi-c1}
Suppose that $f^{-1}(J)$ and $J$ are finitely generated in  $A$ and $f(A)+J$, respectively. Then:
 $A\bowtie^{f}J$ is semihereditary if and only if $A$ is semihereditary, $f(A)+J$ is coherent  arithmetical (resp., Coherent Gaussian), $J\cap \nil(B)=0$ and, for every $\m\in\Max(A,f^{-1}(J))$, $J_{S_{\m}}=0$ or $f_{\m}^{-1}\left(J_{S_{\m}}\right)=0$.
\end{corollary}

See Example~\ref{weak-ex1} where  $A\bowtie^{f}J$ is  semihereditary and $f(A)+J$ is not  semihereditary. Corollary~\ref{semi-c1} recovers a known result for duplications:

\begin{corollary}[{\cite[Theorem 4.1(2)]{CJKM}}]\label{semi-c1-1}
Suppose $I$ is finitely generated. Then:\\
$A\bowtie I$ is semihereditary if and only if $A$ is semihereditary and $I_{\m}=0,\ \forall\ \m\in\Max(A,I)$.
\end{corollary}

Contrast this result with \cite[Corollary 4.15]{Fin} on the transfer of the semihereditary property to amalgamations.

\section{Examples}\label{sec:3}

First, as an illustrative example for Theorem~\ref{arith}, we provide a family of non-reduced arithmetical rings which arise as bi-amalgamations.


\begin{example}\label{arith-ex1}
Let $(A, \m)$ be a valuation domain, $K:=\qf(A)$, $E$ a finitely generated $A-$module such that $E_{\m}=0$, and $B:=A\ltimes E$ the trivial ring extension of $A$ by $E$, and $C:= K[[X]]$. Consider the natural injective ring homomorphisms $f: A \hookrightarrow B$ and $g: A \hookrightarrow C$ and let $J:=0\ltimes E$.  We claim that the bi-amalgamation $R:=A\bowtie^{f,g}(J,0)$ is a non-reduced arithmetical ring. Indeed, notice first that $f^{-1}(J)=g^{-1}(0)=0$, $f(A)+J=B$, and $g(A)=A$. Further, $B$ is (local) arithmetical by \cite[Theorem 3.1]{MiKabbourMa}. So, $R$ is a chained ring by Theorem~\ref{arith}. However, $R$ is not reduced by \cite[Proposition 4.7]{KLT} as $J^{2}=0$.
\end{example}

Next, as an illustrative example for Corollary~\ref{arith-c1}, we provide a non-reduced arithmetical ring which arises as an amalgamation.

\begin{example}\label{arith-c1-ex1}
Let $(A, \m)$ be a valuation domain, $E$ a nonzero divisible $A$-module whose submodules are totally ordered by inclusion (e.g., $E:=\qf(A)$), and $B:=A\ltimes E$ the trivial ring extension of $A$ by $E$. Consider the natural injective ring homomorphism $f: A \hookrightarrow B$ and let $J:=0\ltimes E$. Then, the amalgamation $R:=A\bowtie^fJ$ is a non-reduced arithmetical ring. Indeed, $f^{-1}(J)=0$ and $f(A)+J=B$ is a chained ring by \cite[Theorem 4.16]{AW}. So, $R$ is a chained ring by Corollary~\ref{arith-c1} but not reduced by \cite[Proposition 5.4]{DFF1} since $J\cap \nil(B)\neq 0$.
\end{example}

The converses of Corollary~\ref{weak} and Corollary~\ref{semi} are not true in general. A counter-example is given below in the trivial case of an amalgamation where $A\bowtie^{f}J\cong A$ and $f(A)+J=B$.

\begin{example}\label{weak-ex1}
Consider the canonical surjective ring homomorphism $f: \Z \rightarrow \Z/4\Z$ and let $J$ denote the zero ideal of $\Z/4\Z$ . Then, $\Z\bowtie^{f}J\cong \Z$ is a Dedekind domain and $f(\Z)+J=\Z/4\Z$ is not reduced so that $\wdim(\Z/4\Z)>1$.
\end{example}

Next, in order to illustrate Corollary~\ref{weak}, one may use bi-amalgamations to enrich the literature with new examples of non-semihereditary rings with weak global dimension $\leq 1$ from the existing ones.

\begin{example}\label{weak-ex2}
Let $A_{o}$ be a Noetherian ring with $\wdim(A_{o})\leq 1$ and let $I$ be a proper ideal of $A_{o}$ such that $I_{\m}=0,\ \forall\ m\in \Max(A_{o},I)$ $\big(e.g., A_{o}:=\Z/12\Z$ and $I:=4\Z/12\Z$; clearly, $I_{\m_{1}}=0$ and $I_{\m_{2}}=0$, where $\m_{1}:=2\Z/12\Z$ and $\m_{2}:=3\Z/12\Z\big)$. Let $A:=A_{o}\bowtie I$ be the amalgamated duplication of $A_{o}$ along $I$. By Corollary \ref{arith-c3-1}, $\wdim(A)\leq 1$. Let $D$ be any non-coherent ring with $\wdim(D)\leq 1$ (e.g., \cite[Example 4.1]{GS}). Finally, consider the natural ring homomorphisms $f: A \twoheadrightarrow A_{o}$ and $g: A \hookrightarrow A\times D$, and let $J:=I$ and $J':=(I\bowtie I)\times D$. Then, the bi-amalgamation $R:=A\bowtie^{f,g}(J,J')$ is a non-semihereditary ring with weak global dimension $\leq 1$. Indeed, notice first that $f^{-1}(J)=g^{-1}(J')=I\bowtie I$. Then, we have $\wdim(f(A)+J)=\wdim(A_{o})\leq 1$ and $\wdim(g(A)+J')=\wdim((A\times0)+ ((I\bowtie I)\times D))=\wdim(A\times D)=\sup\{\wdim(A),\wdim(D)\}\leq 1$. Moreover, let $\m\bowtie I\in\Max(A,I\bowtie I)$. Necessarily, $\m\in\Max(A_{o},I)$. Therefore, $S_{\m\bowtie I}=f(A - (\m\bowtie I))+J = (A_{o}- \m) + I$ and hence $J_{S_{\m\bowtie I}}=I_{\m}=0$. Now, $A_{o}$ and $D$ are reduced and so are $A$ and $A\times D$. By Corollary~\ref{weak}, $\wdim(R)\leq 1$. Finally, note that $R$ is not coherent (and, a fortiori, not semihereditary) since $\frac{R}{J\times 0}\cong g(A)+J'=A\times D$ is not coherent (as $D$ is not coherent).
\end{example}

Next, as an illustrative example for Corollary~\ref{semi}, we provide a new example of semihereditary ring which arises as a bi-amalgamation.


\begin{example}\label{semi-ex1}
Let $A$ be a Noetherian semihereditary ring and let $I$ be a proper ideal of $A$ such that $I_{\m}=0,\ \forall\ m\in \Max(A,I)$ $\big($e.g., $A:=\Z/12\Z$ and $I:=4\Z/12\Z\big)$. Let $B:=A\bowtie I$ be the amalgamated duplication of $A$ along $I$ and let $D$ be a Noetherian semihereditary ring. Finally, consider the natural injective ring homomorphisms $f: A \hookrightarrow B$ and $g: A \hookrightarrow A\times D$, and let $J:=I\bowtie I=I\times I$ and $J':=I\times D$. Then, the bi-amalgamation $R:=A\bowtie^{f,g}(J,J')$ is a Noetherian semihereditary ring. Indeed, note that $f^{-1}(J)=g^{-1}(J')=I$ and, $\forall\ \m\in\Max(A,I)$, $S_{\m}:=f(A-{\m})+I\times I=\big((A-{\m})+I\big)\times \big((A-{\m})+I\big)$. So that $J_{S_{\m}}=(I\times I)_{S_{\m}}\cong I_m\times I_{\m}=0.$ Moreover, $f(A)+J=A\bowtie A=A\times A$ and $g(A)+J'=A\times 0+I\times D=A\times D$ are Noetherian semihereditary rings (since so are $A$ and $D$). Therefore, Corollary~\ref{semi} leads to the conclusion.
\end{example}


\end{document}